\documentclass[11pt,a4paper]{amsart}

\usepackage{amssymb}
\usepackage{amscd}
\usepackage{fullpage}
\usepackage[OT4]{fontenc}

\theoremstyle{definition}
\newtheorem{definition}{Definition}
\theoremstyle{plain}
\newtheorem{theorem}[definition]{Theorem}
\newtheorem{lemma}[definition]{Lemma}

\theoremstyle{remark}

\newenvironment{varthm*}[1]{\trivlist\item[]{\bf #1.}\it}{\endtrivlist}

\DeclareMathOperator{\reg}{reg}

\DeclareMathOperator{\charf}{char}

\let\to\longrightarrow

\def\PP{\mathbb{P}}
\def\FF{\mathbb{F}}

\let\to\longrightarrow

\def\wdc{\widehat{\alpha}}

\begin{document}

\title{A containment result in $\PP^n$ and the Chudnovsky conjecture}
\author{Marcin Dumnicki, Halszka Tutaj-Gasi\'nska}
\date{}

\dedicatory{
Faculty of Mathematics and Computer Science, Jagiellonian University, \\
ul. \L{}ojasiewicza 6, 30-348 Krak\'ow, Poland \\
Email addresses: Marcin.Dumnicki@im.uj.edu.pl\\
Halszka.Tutaj-Gasinska@uj.edu.pl
}



\begin{abstract}

In the paper we prove the containment $I^{(nm)}\subset M^{(n-1)m}I^m$, for a radical ideal $I$ of $s$ general points in $\PP^n$, where $s\geq 2^n$. As a
corollary we get that the Chudnovsky Conjecture holds for a very general set of at least $2^n$ points in $\PP^n$.
\end{abstract}

\maketitle
\thispagestyle{empty}
\pagestyle{empty}

\section{Introduction}

   Given any subscheme $Z \subset \PP^n$ and its homogenous ideal
   $I=I_Z$ in $\FF[\PP^n]=\FF[x_0,\dots,x_n]$, we define $\alpha(I)$ as the minimal degree of a non-zero
   element in $I$. We will assume that $\charf \FF = 0$.

   In general $\alpha(I)$ is hard to compute and it behaves quite
   unpredictably. However there is an asymptotic counterpart of the
   $\alpha$-invariant, which is
   the \emph{Waldschmidt constant} (\cite{Wal77}) defined as
   \begin{equation}\label{eq:Waldschmidt}
      \wdc(I) = \lim_{m \to \infty} \frac{\alpha(I^{(m)})}{m},
   \end{equation}
   where $I^{(m)}$ denotes the $m$-th symbolic power of the ideal $I$ (for the definition and basic properties
of $I^{(m)}$ see \cite{HaHu}).
It turns out that this constant is well defined and satisfies the
inequality
\begin{equation}\label{wdcin}
\alpha(I^{(m)}) \geq m\wdc(I)
\end{equation}
for all $m$. In fact $\wdc(I)=\inf\limits_{m\geq  1}\frac{\alpha(I^{(m)})}{m}$.

For an ideal $I$ we have $\alpha(I^r)=r\alpha(I)$, but the behaviour of $\alpha(I^{(m)})$ is much more complicated and less understood. Skoda in 1977
\cite{Skoda} showed that $\alpha(I^{(m)}) \geq m\alpha(I)/n$ for an ideal $I$ of points in $\PP^n$ (over complex numbers). Chudnovsky \cite{Ch}, in 1981, improved that bound for $n=2$ to $\alpha(I^{(m)}) \geq m(\alpha(I)+1)/2$ and conjectured the following.

\begin{varthm*}{Chudnovsky Conjecture}\label{chudn}
For an ideal $I$ of points in $\PP^n$ the following inequality holds:
 $$\frac{\alpha(I^{(m)})}{m} \geq \frac{\alpha(I)+n-1}{n}.$$
In particular
 $$\wdc(I) \geq \frac{\alpha(I)+n-1}{n}.$$
\end{varthm*}

Esnault and Viehweg \cite{EV83}, in 1983 showed that $\alpha(I^{(m)})
\geq m(\alpha(I)+1)/n$ for any set of points in $\PP^n$.

In 2002 Ein, Lazarsfeld, Smith in \cite{ELS} and Hochster and Huneke in  \cite{HoHu02} showed that, for any homogeneous ideal $I$ in $\FF[\PP^n]$, the
containment $I^{(nm)} \subset I^m$ holds, thus recovering the Skoda bound more generally --- for all homogeneous ideals
(since $I^{(nm)} \subset I^m$, $\alpha(I^{(nm)}) \geq \alpha(I^m)=nm\alpha(I)/n$; passing with $m$ to infinity
gives $\wdc(I) \geq \alpha(I)/n$). Harbourne and Huneke in \cite{HaHu},
Lemma 3.2, observed that Chudnovsky Conjecture would follow from a more general containment
\begin{equation}\label{maxcont}
I^{(nm)} \subset M^{(n-1)m}I^m,
\end{equation}
where by $M=(x_0,\dots,x_n)$ we denote the irrelevant maximal ideal.

The containment \eqref{maxcont} holds (for a given $m$) for general points in $\PP^2$ (\cite{HaHu}, Proposition 3.10), for at most $n+1$ general points in
$\PP^n$ and also for general points in $\PP^3$ (\cite{Du}, \cite{Dumxxx}). As a corollary, the Chudnovsky Conjecture holds for very general points in $\PP^3$.

The main result of the present paper is the following theorem.

\begin{theorem}
For a nonnegative integer $m$, and for a radical ideal $I$ of $s$ general points in $\PP^n$,
where $s\geq 2^n$, the containment
$$I^{(nm)}\subset M^{(n-1)m}I^m$$
holds.

As a corollary, the Chudnovsky Conjecture holds for a very general set of at least $2^n$ points in $\PP^n$.
\end{theorem}
We will work on filling the gap between $n+1$ and $2^n$ in our future project.

\section{A bound for $\wdc(I)$}

In this section we give a bound for the
Waldschmidt constant of an ideal of $s$ very general points in
$\mathbb{P}^n$. The bound in fact easily follows from the much stronger result of Evain
\cite{evain}, who showed that for an ideal $I$ of general $k^n$ points $\alpha(I)$ is ``expected''.
Since the methods of Evain are highly non-trivial and very delicate, we give a short proof of our bound here to make the paper more
self-contained.

\begin{theorem}\label{bound}
For a radical ideal $I$ of $s$ very general  points $P_1,\dots ,P_s$ in $\mathbb{P}^n$ we have
$$\wdc(I)\geq \lfloor \sqrt[n]{s} \rfloor.$$
\end{theorem}

Proof. To prove the bound we have to show that the system of divisors of  degree $dm-1$ passing through $P_1,\dots ,P_s$ with multiplicity $m$ is empty if $d =
\lfloor\sqrt[n]{s}\rfloor$.  By $(d, m^{\times s})$ we denote a system of divisors of degree $d$, passing through $s$ (general) points with multiplicity $m$.

Without loss of generality we may suppose that $s=k^n, k\in \mathbb{N}$, as the emptiness of the system $(km-1; m^{\times k^n})$, implies the emptiness of
$(km-1, m^{\times r})$, $k^n\leq r < (k+1)^n$.

For $n=1$ the nonexistence of  the system $(km-1; m^{\times k})$ is trivial. Then we proceed by induction.  (Note that for $n=2$ the nonexistence of  the
system $(km-1; m^{\times k^2})$ was also proved by Nagata in \cite{Nag59}).

 For $n\geq 2$,  suppose there exists a divisor in
$\mathbb{P}^n$ of degree $km-1$, passing through $P_1,\dots ,P_s$ with multiplicity $m$.  Take $k$ general hyperplanes in $\mathbb{P}^n$ and put $k^{n-1}$
points on each such hyperplane (in general position). Then on the hyperplane we have the system of divisors of degree $km-1$ which have to pass through
$k^{n-1}$ points with multiplicity $m$. By inductive assumption, this system is empty. Thus, all hyperplanes must be components of the system. Repeating the
procedure (of checking that all hyperplanes must be the components of the system)  $m$ times we get that the system has to have degree $km$, not $km-1$, a
contradiction.

\section{A combinatorial lemma}
Here we prove an auxiliary combinatorial lemma, which we will use
in the proof of Theorem~\ref{containment theorem}

\begin{lemma}\label{combinatorial}
(1) For any integers $k\geq 4$ and $n\geq 3$ the following
inequality holds:
$$ k^n\leq \binom{kn-n}{n}.$$
(2) Moreover, if $n\geq 5$
$$3^n\leq \binom{2n}{n}$$
holds.
\end{lemma}

\begin{proof}
First we prove (1). Observe that it is enough to prove the
inequality for $k=4$ only. Indeed, our inequality may be written
as
$$ (nk-n)\cdot \ldots \cdot (nk-2n+1) \geq n! k^n.$$
The left-hand side is greater or equal to
$$\left(n-\frac{n}{4}\right)k\cdot \ldots \cdot \left(n-\frac{2n-1}{4}\right)k,$$
as $k\geq 4$. Dividing both sides by $k^n$ and multiplying by
$4^n$ we see that this is enough to prove
$$(4n-n)\cdot \ldots \cdot (4n-2n+1)\geq n!4^n,$$
i.e. our inequality with $k=4$.

Thus, we have to prove
$$ \binom{3n}{n}\geq 4^n.$$
Since
$$ \frac{\binom{3(n+1)}{n+1}}{\binom{3n}{n}} = \frac{(3n+3)(3n+2)(3n+1)}{(n+1)(2n+2)(2n+1)} > 4$$
for all $n \geq 1$ and
$$\binom{9}{3} > 4^3,$$
the claim follows by induction.

As for the second claim of the lemma, we proceed analogously, observing that
$$ \frac{\binom{2(n+1)}{n+1}}{\binom{2n}{n}} = \frac{(2n+2)(2n+1)}{(n+1)(n+1)} > 3$$
for all $n \geq 1$ and
$$\binom{10}{5} > 3^5.$$
\end{proof}

\section{A containment result}
Now we are able to prove our containment theorem.

\begin{theorem}\label{containment theorem}
For a radical ideal $I$ of $s$ very general  points in
$\mathbb{P}^n$, where $s\geq 2^n$ and $n\geq 3$ the containment
    $$I^{(n m)}\subset M^{(n-1)m}I^m$$
holds, for any integer $m\geq 1$.
\end{theorem}

Proof.   Dumnicki in \cite{Du} showed the containment for any number of
very general points in $\mathbb{P}^3$, so we assume $n \geq 4$.

Let $\reg(I)$ denote the Castelnuovo-Mumford regularity of $I$, see e.g. \cite{DeSi}, let $\sigma(I)$
denote the maximal degree of an element in a minimal set of generators of $I$.
Since in general $\sigma(I) \leq \reg(I)$, by \cite[Proposition 2.3.]{HaHu} it is enough to show that
$$\alpha(I^{(mn)}) \geq (n-1)m+m\reg(I).$$
By \eqref{wdcin} the above inequality follows from
$$n\wdc(I)\geq n-1 + \reg(I).$$

From Theorem \ref{bound} we know that
    $$\wdc(I)\geq \lfloor\sqrt[n]{s}\rfloor.$$
For the ideal $I$ as in our theorem we have that $\reg(I)=r+1$, where $r$ is such an integer that
    $$\binom{r-1+n}{n}< s\leq \binom{r+n}{n},$$
see \cite{BocHar10b}, page 2.

 Thus, it is enough to prove that
    \begin{equation}\label{ineq}
        n\lfloor\sqrt[n]{s}\rfloor\geq n-1+r+1,
    \end{equation}
 for $r$ satisfying

    $$\binom{r-1+n}{n}< s.$$

Suppose to the contrary, that $n\lfloor\sqrt[n]{s}\rfloor< n-1+r+1$. If we can prove that then
    \begin{equation}\label{nierownosc-z-s}
    \binom{r+n-1}{n}\geq s,
\end{equation}
     we get a contradiction, and we are done.

As $\lfloor\sqrt[n]{s}\rfloor $ is constant (equal $k-1$) for all $s=(k-1)^n,\dots,k^n-1$, we may assume that the right hand side of inequality
(\ref{nierownosc-z-s}) equals $k^n-1$, which is the worst possible case.

Thus, we reduced the problem to proving that for $r\geq nk-2n+1$ we have
    $$k^n-1\leq \binom{r+n-1}{n}.$$
This will follow if we prove that
    $$ k^n\leq \binom{kn-n}{n}.$$
This last inequality is proved in Lemma \ref{combinatorial} for  $n=4$ and $k\geq 4$ and for $n\geq 5$ and $k\geq 3$. This proves our theorem for all $n\geq 5$ and $s\geq 2^n$ (remember, that we have $s\geq (k-1)^n$) and also when $n=4$ and $s\geq 3^n$.
Moreover, by direct computation inequality
\eqref{ineq} holds for $n=4$ with $2^4 \leq s \leq 70$, so it remains only to prove our theorem only for $n=4$ and $71 \leq s < 3^4$. This is done in the lemma below.

\begin{lemma}
The Waldschmidt constant for ideal of at least 71 very general points in $\PP^4$ is bounded  below by $9/4$.
\end{lemma}

Before the proof, let us observe that using the bound $9/4$ instead of $\lfloor\sqrt[4]{s}\rfloor$ in \eqref{ineq} finishes the proof for $n=4$ and
$s=71,\dots,80$.

\begin{proof}
It is enough to show that for each $m$ the system of hypersurfaces of degree $9m-1$ with 71 $4m$-fold points is empty. Since the system of degree $8m-1$ with
16 $4m$-fold points is empty, by \cite{DA} (Theorem 1) it is enough to show that $9m-1$ with four $8m$-fold points and seven $4m$-fold points is empty. The
last one, by Theorem 3  in \cite{DA}, is equivalent to a system of degree $2$ with at least one point of multiplicity $4m$, which is empty.
\end{proof}

\section*{Acknowledgements}
We are grateful to Tomasz Szemberg and anonymous referees who helped us to improve the paper.\\
The  research was partially supported by the National Science Centre, Poland, grant 2014/15/B/ST1/02197.

\end{document}